\documentclass[11pt,a4paper,english]{article}

\usepackage[a4paper,margin=2.5cm,vmargin={1.5cm,2.5cm},includeheadfoot]{geometry}
\usepackage[scaled=0.92]{helvet}
\usepackage[latin 1]{inputenc}

\usepackage{amssymb,amscd, amsmath, babel, amstext, amsthm, latexsym}
\usepackage{color}
\usepackage[dvipsnames]{xcolor}
\usepackage{breqn, dsfont, graphics, enumerate}
\usepackage[font=scriptsize]{caption}
\usepackage{epstopdf}
\usepackage{graphicx}
\usepackage{mathrsfs}
\usepackage[toc,page]{appendix}

\usepackage{enumitem}

\theoremstyle{definition}
\newtheorem{definition}{Definition}

\numberwithin{equation}{section}
\numberwithin{definition}{section}
\numberwithin{definitions}{section}
\numberwithin{thm}{section}
\numberwithin{lem}{section}
\numberwithin{remark}{section}

\newcommand{\dlambda}{\ensuremath{\text{d}\lambda}}
\newcommand{\dM}{\text{d}M}
\newcommand{\domega}{\ensuremath{\text{d}\omega}}

\newcommand{\dS}{\text{d}S}

\newcommand{\dW}{\text{d}W}

\newcommand{\ud}{\text{d}}
\newcommand{\IN}{{\mathbb{N}}}
\newcommand{\IR}{{\mathbb{R}}}
\newcommand{\B}{\ensuremath{\mathcal{B}}}
\newcommand{\id}{\text{id}}
\newcommand{\indep}{\rotatebox[origin=c]{90}{$\models$}}
\newcommand{\ind}{\mathbf{1}}
\newcommand{\E}{\ensuremath{\mathcal{E}}}
\newcommand{\F}{\ensuremath{\mathcal{F}}}
\newcommand{\G}{\ensuremath{\mathcal{G}}}
\newcommand{\Hh}{\ensuremath{\mathcal{H}}}

\newcommand{\FF}{\mathbb{F}}
\newcommand{\GG}{\mathbb{G}}
\newcommand{\HH}{\mathbb{H}}

\newcommand{\PP}{\mathbb{P}}
\newcommand{\EE}{\mathbb{E}}

\DeclareMathOperator*{\esssup}{ess\,sup}

\newcommand{\R}{\mathbb{R}}

\newcommand{\A}{\ensuremath{\mathcal{A}}}

\newtheorem{assumption}{Assumption}

\newtheorem*{summary*}{Summary}

\newtheorem{theorem}{Theorem}

\newtheorem{lemma}[theorem]{Lemma}

\newtheorem{proposition}[theorem]{Proposition}

\newtheorem*{remarks*}{Remarks}
\newtheorem*{remark*}{Remark}

\numberwithin{theorem}{section}
\numberwithin{remarks}{section}

\newcommand{\Ff}{\mathcal{F}}

\newcommand{\N}{\mathcal{N}}

\definecolor{mygray}{gray}{0.6}
\definecolor{mycoral}{rgb}{0.25, 0.0, 1.0}

\title{Utility maximisation and change of variable formulas for time-changed dynamics}

\author{Giulia Di Nunno\thanks{Department of Mathematics,
University of Oslo, P.O. Box 1053 Blindern, N-0316 Oslo, and Department of Business and Management Science, NHH, Helleveien 30, N-5045 Bergen, Norway.
Email: giulian@math.uio.no}  \and
Hannes Haferkorn\thanks{Commerzbank AG. Email: hanneshagen.haferkorn@commerzbank.com}
\and
Asma Khedher\thanks{Korteweg-de Vries Institute for Mathematics, P.O. Box 94248, 1090 GE Amsterdam, The Netherlands. Email: a.khedher@uva.nl}
\and
Mich\`ele Vanmaele\thanks{Department of Applied Mathematics, Computer Science and Statistics, Ghent University, Krijgslaan 281-S9, Gent, 9000, Belgium.
Email: Michele.Vanmaele@UGent.be}
}
\date{July 3, 2024}

\begin{document}
\maketitle

\begin{abstract}
	In this paper we derive novel change of variable formulas for stochastic integrals w.r.t.\ a time-changed Brownian motion where we assume that the time-change is a general increasing stochastic  process with finitely many jumps in a bounded set of the positive half-line and is independent of the Brownian motion. 
	As an application we consider the problem of maximising the expected utility of the terminal wealth in a semimartingale setting, where the semimartingale is written in terms of a time-changed Brownian motion and a finite variation process. To solve this problem, we use an initial enlargement of filtration and our change of variable formulas to shift the problem to a maximisation problem under the enlarged filtration for models driven by a Brownian motion and a finite variation process. 
The latter problem can be solved by using martingale properties. 
Then applying again a change of variable formula, we derive the optimal strategy for the original problem for a power utility and for a logarithmic utility.

\end{abstract}


\section{Introduction}
{\it Time-change} is a modelling technique that allows to change the speed at which a process runs through its paths. See, e.g., \cite{jacod2006calcul, el1977theorie} for an overview. Time-changed semimartingales are well studied in the literature in the cases when the time-change is absolutely continuous w.r.t.~the Lebesgue-measure or when it is a subordinator (i.e., an increasing L\'evy process). See, e.g., \cite{kallsen2002time, sauri2017class, di2014bsdes, diNunno, kallsen2010utility}. 
In this paper, we consider a {\it time-changed Brownian motion}  $(M_t)_{0\leq t\leq T}$, $M_t:=W_{\Lambda_t}$, where we assume that the time-change $(\Lambda_t)_{0\leq t\leq T}$ is a general increasing stochastic process with finitely many jumps in a bounded set of the positive half-line and it is independent of the Brownian motion $(W_t)_{0\leq t \leq T}$.
In this way, we do allow $(\Lambda_t)_{0\leq t \leq T}$ to jump, though being outside of the framework of L\'evy processes.

\par \medskip
Our motivation for looking at such time-changed noises lies in the fact that they reflect well some of the stylised facts observed in real financial data and yet they are quite statistically tractable models providing also good calculus potential in the stochastic analysis.
Indeed, the time-change offers a very natural way to introduce stochastic volatility in the model of  risky asset prices: The market time' $\Lambda_t$ is -- in contrast to the physical time $t$ -- linked to the number of trades and is as such reflecting the flow of news on the market. The more trades happen at a fixed physical time interval $(t_0,t_0+\varepsilon)$, the faster the market time evolves (relative to physical time), i.e., the steeper the function $t\mapsto\Lambda_t$ is on $t\in(t_0,t_0+\varepsilon)$. 
Jumps of the time-change $(\Lambda_t)_{0\leq t \leq T}$ are to be interpreted as an explosion of the number of trades, which typically happens when some critical news is coming in. See, e.g., \cite{BNS-book, S-book, GMY-2001}.

\par \medskip
Our \textit{first main contribution} is the derivation  of novel change of variable formulas for stochastic integrals w.r.t.\ a time-changed Brownian motion. We start from $(M_t)_{0\leq t\leq T}$ and a filtration $\FF= (\Ff_t)_{0\leq t\leq T}$ generated by $(M_t)_{0\leq t\leq T}$ and $(\Lambda_t)_{0\leq t\leq T}$ under which  $(W_t)_{0\leq t\leq T}$ is not necessarily a Brownian motion. We introduce the 
 \textit{enlarged filtration} $\HH=(\Hh_t)_{0\leq t\leq T}$ given by
\begin{align*}
\Hh_t:=\F^W_t\vee\F^\Lambda_T,
\end{align*}
i.e., $\Hh_0$ contains already all the information about the entire future of the time-change $(\Lambda_t)_{0\leq t\leq T}$ (see, e.g., \cite{jeulin2006semi, jacod2006calcul} for more about enlargement of filtrations).
The advantage from enlarging the filtration is that $(W_t)_{0\leq t\leq T}$ is a Brownian motion as we work under the filtration $\HH$. This allows us to prove change of variable formulas to turn the stochastic integral w.r.t.~the time-changed Brownian motion $W_{\Lambda_\cdot}$ into an integral w.r.t.~the Brownian motion $W$
and conversely, i.e., turn the stochastic integral w.r.t.~the Brownian motion $W$ into a stochastic integral w.r.t.~the time-changed Brownian motion $W_{\Lambda_\cdot}$. Our results are substantially different from the change of variable formulas available in the literature, which deal with stochastic integrals w.r.t.\ time-changed \emph{$\Lambda$}-adapted semimartingales, i.e., semimartingales  which are constants on all the stochastic time intervals $[\Lambda_{t-},\Lambda_{t}]$, $t\in[0,T]$, see \cite{kobayashi2011stochastic} and \cite{jacod2006calcul}. Indeed, our formula deals with the time-changed Brownian motion as integrator, which is not necessarily $\Lambda$-adapted.

\par \medskip

Our \textit{second main contribution} is to solve a utility maximisation problem from terminal wealth, where the dynamics are driven by the time-changed Brownian motion. We consider the filtration $\FF$ and  a controlled stochastic dynamics $(V_t^\nu)_{0\leq t \leq T}$ driven by a semimartingale $S_t=M_t +A_t$, $0\leq t\leq T$, where $M_t= W_{\Lambda_t}$ and $(A_t)_{0\leq t\leq T}$ is a c\`adl\`ag process of finite-variation. Here we impose on $\Lambda$ to be {\it strictly} increasing with finitely many jumps in a bounded set of the positive half-line.
 In particular, we are interested in solving the utility maximisation problem with objective function 
\begin{align*}
 J(\nu):=\mathbb{E}[U(V^{\nu}_T)|\F_t]\,,
\end{align*}
for a utility function $U$ over the set of admissible $\FF$-adapted controls $\nu$, for every $t\in[0,T]$. 
Later on, we specify the type of utility function to be a \textit{power} or a \textit{logarithmic utility}.
The maximisation of expected utility from terminal wealth is a classical problem in mathematical finance (we refer, e.g.,~to \cite{karatzas1998methods} for an overview). Different approaches are used in the literature to solve such a problem relying on the theory of partial differential equations, duality characterisations of portfolios, or the theory of quadratic backward stochastic differential equations, see, e.g., \cite{merton1975optimum, benth2003merton, cox1991variational, karatzas1991martingale, he1991consumption, kramkov1999asymptotic, hu2005utility, morlais2010new}, where the maximisation problem is considered in continuous and jump settings. The case where the price process is modelled by a time-changed L\'evy process with the time-change being absolutely continuous w.r.t.~the Lebesgue-measure is considered in \cite{di2014bsdes, diNunno, kallsen2010utility}. 
To the best of our knowledge, we are the first to tackle the optimisation problem above in the setting of time-changed Brownian motion for general time-changes, which are increasing and allowed to jump at finitely many times. 
This setting entails some challenges which we approach by interplaying with different information flows and exploiting the change of variable formulas we have introduced.

\par \medskip
Indeed we study, at first, the optimisation problem in the framework with the \textit{enlarged filtration} $\HH=(\Hh_t)_{t\geq 0}$, which we time-change. 
The objective is to maximise the performance
\begin{align*}
 J_\HH(\nu):=\mathbb{E}[U(V^{\nu}_T)|\Hh_{\Lambda_t}]\,,
\end{align*}
over all $\mathbb{F}$-adapted and admissible controls $\nu$. We prove that an optimiser in this setup yields an upper bound for the solution to the original problem (i.e., when we condition on $\F_t$ instead of $\Hh_{\Lambda_t}$), that follows from
\[
\esssup_{\nu\in\A_{t;\FF}}J_\HH(\nu) \leq \esssup_{\tilde{\nu}\in\tilde\A_{\Lambda_t;\HH}}J_\HH(\tilde{\nu}),
\]
where $\A_{t;\FF}$ is the set admissible $\mathbb{F}$-adapted strategies and $\tilde\A_{\Lambda_t;\HH}$ is the set admissible $\mathbb{H}$-adapted strategies.

Then using the change of variable formulas we introduce in this work and adapting the approach by \cite{kallsen2010utility} and  \cite{GOLL2000}, we obtain explicit expressions for the optimal strategy and the optimal value function for the {\it power} and the {\it logarithmic utility} functions under some conditions on the finite-variation part of our considered model. The conditions imposed on the model permit to preserve some properties of the optimal strategy, in particular some measurability properties, that might otherwise be lost after the application of the change of variable formulas. So we obtain
\[
J_\HH(\hat{\nu})=\esssup_{\nu\in\A_{t;\FF}}J_\HH(\nu)= \esssup_{\tilde{\nu}\in\tilde\A_{\Lambda_t;\HH}}J_\HH(\tilde{\nu}),  \;\;\mbox{and\ }\; J(\hat{\nu}) =\mathbb{E} \left[J_\HH(\hat{\nu})\mid \mathcal{F}_t\right]\,.
\]
\medskip
This article is organised as follows. In Section \ref{S2} we review the definition and some classical results on time-change and change of variable formulas for stochastic integration. We introduce our framework and show some results on enlargement of filtrations and on the generalised-inverse of the time-change. Section \ref{S4} contains our main results on the change of variable formulas for stochastic integrals w.r.t.\ the time-changed Brownian motion when the time-change is a general increasing process allowed to jump at finitely many times. In Section \ref{S5} we use these formulas and exploit different information flows to solve first the problem under the enlarged filtration and to draw conclusions for our original problem. 
In the Appendix we collect some technical results and some of the proofs of Section \ref{S2}.

\section{Framework, enlargement of filtration and the generalised inverse of the time-change}\label{S2}
Fix $T$, $R \in \mathbb{R}_+$. Let $(\Omega,\F,\mathbb{P})$ be a probability space.  
Let $\mathcal{S},\mathcal{T} : \Omega \rightarrow [0,\infty]$. By abuse of notation, we also denote by $[\mathcal{S},\mathcal{T}]$ the stochastic interval, i.e., $[\mathcal{S},\mathcal{T}] = \{(\omega, t) \in \Omega \times [0,\infty];\, \mathcal{S}(\omega) \leq t\leq \mathcal{T}(\omega)\}$. 
Let $\N$ be the set of $\mathbb{P}$-null events. Given a stochastic process $X=(X_r)_{0\leq r \leq R}$, we denote by
\begin{align}\label{natural-filtration}
\F_r^X = \sigma(X_u, \, u\leq r)\vee \N\,, \qquad 0\leq r \leq R\,,
\end{align}
 the augmented sigma-algebra generated by $X$ up to time $r$ and we set $\mathbb{F}^X = (\F_r^X)_{0\leq r\leq R}$. 

In this section, 
 we set the framework and present some preliminary results on enlargement of filtrations and on the generalised-inverse of the time-change. 
Here we consider the product $(\Omega,\F,\PP)$ of two complete probability spaces $(\Omega_\Lambda,\F_\Lambda, \PP_\Lambda)$ and $(\Omega_W,\F_W,\PP_W)$. Namely, 
\begin{align}\label{product-probability}
 \Omega&=\Omega_\Lambda\times\Omega_W\,,\nonumber\\
 \F&=(\F_\Lambda\otimes\F_W)\vee\N\,, \\
 \PP&=\PP_\Lambda\otimes \PP_W\,, \nonumber
\end{align}
We consider a Brownian motion $W$ and an increasing right-continuous stochastic process $\Lambda$, with $\Lambda_0=0$, as the following measurable mappings on $\Omega$:
\begin{equation}\label{W-Lambda}
\begin{aligned}
\Lambda:& \;[0,T]\times\Omega\rightarrow[0,R]\,,  & (t,\omega_\Lambda,\omega_W) &\longmapsto \Lambda_t(\omega_\Lambda)\,, \\
W:& \:[0,R]\times\Omega\rightarrow\IR\,,  & (r,\omega_\Lambda,\omega_W) &\longmapsto W_r(\omega_W),
\end{aligned}
 \end{equation}
namely, $\Lambda$ is $\B([0,T])\otimes\F_\Lambda\otimes\{\emptyset,\Omega_W\}-\B([0,R])$-measurable and $W$ is $\B([0,R])\otimes\{\emptyset,\Omega_\Lambda\}\otimes\F_W-\B(\IR)$-measurable. Observe that $W$ and $\Lambda$ are independent.

\medskip
Let $\mathbb{F}^\Lambda$ and $\mathbb{F}^W$ be the filtrations generated by $\Lambda$ and $W$, respectively, see \eqref{natural-filtration}.
 We assume that
\begin{align}\label{Assumption 1}
\F^\Lambda_T=(\F_\Lambda\otimes\{\emptyset,\Omega_W\})\vee\N\,, \nonumber \\
\F^W_{R}=(\{\emptyset,\Omega_\Lambda\}\otimes\F_W)\vee\N\,.
\end{align}

\begin{proposition}\label{AufteilungSigmaAlgF}
In the framework \eqref{product-probability}-\eqref{Assumption 1}, we have that:
\begin{itemize}
\item[i)]
$\F=\F^\Lambda_T\vee\F^W_{R}$,
\item[ii)]
$\F^\Lambda_T$ is independent of $\F^W_{R}$.
\end{itemize}
\end{proposition}
Observe that $\Lambda$ is an $\FF^\Lambda$-{\it time-change}. In the sequel we consider {\it a time-changed Brownian motion} $(M_t)_{0\leq t\leq T}$, i.e.,
\begin{align}\label{time-changed-bm}
M_t := W_{\Lambda_t}\,, \qquad t \in [0,T]\, .
\end{align}
%
Time-changed Brownian motions are widely used in finance to model log stock returns and entail a big class of processes that go beyond Brownian motion both within continous models and including jumps, these account for both time clustering and stochastic volatility.
See, e.g. \cite{Barndorff-Nielsen2002}, and the comprehensive books \cite{BNS-book, Cont-Tankov, S-book}. 
For example, when the time-change is a L\'evy subordinator, one obtains the well known variance gamma (VG) model and the normal inverse Gaussian (NIG) model (see, e.g.,~\cite[Chapter 6]{sato1999levy}).
Also another area of use of time-changed models is in modelling turbulence. See, e.g., \cite{Schmiegel}. 

We stress that, using time-changed noises, one can have both Markovian and non-Markovian structures, which give a large flexibility from the modelling point of view. For example, subordinating a Brownian motion provides Markovianity, while using an absolutely continuous type of time-change provides, in general, a non-Markovian process.
Given the generality of the time-change we consider in this work, we allow for large flexibility and we work exploiting different information flows, under which we consider different properties of the processes.

\par \medskip

Indeed, we specify the $\mathbb{P}$-augmented filtration $\FF= (\F_t)_{t\in[0,T]}$ in connection with the time-changed Brownian motion $M$, with $M_t(\omega)= M_t(\omega_\Lambda, \omega_W) := W_{\Lambda_t(\omega_\Lambda)}(\omega_W)$, as
  \begin{align}\label{filtration-M-lambda}
   \F_t&:=\bigcap_{s>t}\left(\F^M_s\vee \F^\Lambda_s\right)\,, \, t \in [0,T)\,,\nonumber\\
    \F_T &= \F^M_T \vee \F^\Lambda_T\,.
  \end{align}
Note that  in our setting  the filtration $\mathbb{F}^M$ does not coincide with $\FF$ as $W$ is not necessarily $\Lambda$-adapted and the results from \cite[Theorem 10.17]{jacod2006calcul} do not hold as illustrated in \cite[Example 2.5]{kobayashi2011stochastic}.
  
\medskip
From now on  we endow the probability space $(\Omega,\F, \PP)$ with the filtration $(\F_t)_{0\leq t \leq T}$. This is possible, since the following observation holds.
\begin{proposition}\label{endowment}
Let $\F$, $\mathbb{F}$ be as in \eqref{product-probability}, \eqref{filtration-M-lambda}, respectively. Then $\F_t\subseteq\F$, $0\leq t \leq T\,.$
 \end{proposition}


In this context the following result is naturally derived.
\begin{proposition}\label{M-F-martingale}
The process $M = W_\Lambda$ is an $\FF$-martingale.
\end{proposition}
The proofs of the results here above are given in the Appendix for completeness.
 Observe that the stochastic process $W$ is not necessarily a Brownian motion under the filtration $\FF$. 
 Here below we introduce a new information flow under which $W$ is a Brownian motion.

\subsection{An enlarged filtration}\label{section:enlargement-of-filtration}
We introduce a new information flow $\HH$ that contains information about the Brownian motion $W$ up to time $t$ and all the information of the time-change $\Lambda$ up to time $T$.  
\begin{definition}\label{def:filtration-H}
The filtration $\HH= (\Hh_r)_{r \in [0,R]}$ is  the {\it initial enlargement} of $\mathbb{F}^W$ by $\F^\Lambda_T$: 
  \begin{align*}\label{filtration-H}
   \Hh_r&:=\F_r^W \vee\F^\Lambda_T \,, \qquad r\in[0,R]\, .
  \end{align*}
\end{definition}

\par \medskip
In general, a martingale will not preserve its martingale property under a larger filtration. Many papers have been dedicated to the study of when this property is preserved, see, e.g., \cite{jeulin2006semi, jacod1985grossissement, stricker1978calcul} and \cite[Chapter VI]{protter2005stochastic}. 
Techniques of enlargement of filtrations have been recently widely used in mathematical finance, in particular, in insider trading models and in models of default risk. It is an important tool in modelling of asymmetric information between different agents and the possible additional gain due to this information (see, e.g.,~\cite{amendinger2000martingale, imkeller1996enlargement, elliott1999incomplete, di2008anticipative, J-book}). In this paper, we use the filtration $\HH$ to prove the change of variable formulas (Theorems \ref{TrafoFormulaEq} and \ref{TrafoFormulaBack}) and we apply this to solve an optimal control problem in Section \ref{section:optimal-control}. The role of information in optimisation problems with time-change was already studied and exploited in \cite{di2014bsdes, diNunno}. There, a maximum principle approach 
was used mixing enlarged filtrations and partial information for time-changed dynamics with an absolutely continuous time-change. This work is then extended in \cite{DNG20} to controlled Volterra type dynamics driven by time-change L\'evy noises.

\par \medskip
Observe that it holds
\begin{equation*}\label{FinH}
 \F_T\subseteq \F = \Hh_{R}\,.
\end{equation*} 



\par \medskip
Hereafter we show two crucial properties for the upcoming applications. Their proofs are presented in the Appendix. 
\begin{proposition}\label{Hcont}
 The filtration $\HH$ is continuous and complete.
\end{proposition}
\begin{proposition}\label{W-BM}
The stochastic process $(W_r)_{r\in[0,R]}$ is an $\HH$-Brownian motion. 
\end{proposition}
Notice that, for any $t$, $\Lambda_t$ is an $\HH$-stopping time. Indeed, for $r \in [0,R]$, we have
\begin{align}\label{stopping-time} 
 \{\Lambda_t\leq r\}\in\F^\Lambda_t\subseteq\F^\Lambda_T\subseteq\Hh_0 \subseteq\Hh_r\,, \qquad t\in [0,T]\,.
\end{align}

Let 
\begin{equation}\label{TCFhat} \hat{\mathbb{H}}=(\hat{\Hh}_t)_{t \in [0,T]}. \qquad \hat{\Hh}_t := \Hh_{\Lambda_t}\end{equation} be the time-changed filtration of $(\Hh_r)_{r \in [0,R]}$, i.e.,
\begin{equation}\label{TCFLambda}\mathcal{H}_{\Lambda_t} = \left\{A \in \Hh_R: A \cap \{\Lambda_t \leq r\} \in \Hh_r, \,\forall r \in [0,R] \right\}\,.\end{equation}


From Proposition \ref{W-BM}, we know that $W$ is an $\HH$-Brownian motion and thus an $\HH$-martingale. The optional sampling theorem yields that
$M$ is an $\hat{\HH}$-martingale.

Then, for all $s \leq t$, $M_s$ is $\hat{\Hh}_{t}$-measurable and, thanks to  \eqref{stopping-time} and the monotonicity of the time-change, also the random variables $\Lambda_s$, for $s\leq t$ are
  $\hat{\Hh}_{t}$-measurable. Then we have
 \begin{align*}
  \F^M_t \vee \F^\Lambda_t \subseteq \hat{\Hh}_{t}\,, \qquad t \in [0,T]\,.
 \end{align*}
Since $\hat{\HH}$ is right-continuous, it holds that  
\begin{equation*}\label{eq:F-in-H-hat}
\F_t\subseteq\hat{\Hh}_{t}, \,\,\text{for all} \,\,t\in[0,T].
\end{equation*}

\subsection{Analysis of the generalised inverse of the time-change}
We present in this subsection some results on the generalised inverse of the time-change $\Lambda$ that we shall need in our derivations later in Section \ref{section:optimal-control}. 
We start by introducing the following definitions.
\begin{definition}\label{def:time-change} \hspace{0.1cm}
\begin{enumerate}
\item The {\it first hitting time} process or {\it generalised inverse} $(\Gamma)_{0\leq r\leq R}$ of the time-change $(\Lambda_t)_{0\leq t \leq T}$ is defined as 
the mapping $\Gamma: \;[0,R]\times\Omega\rightarrow[0,T],$ such that
\label{TC-inverse}
$$
\Gamma (r,\omega_\Lambda,\omega_W) = \Gamma_r(\omega_\Lambda)= \left\{
    \begin{array}{ll}
        \inf\{t;\, \Lambda_t >r\} & \mbox{if } r \in [0,\Lambda_T)\,, \\
        T & \mbox{if } r \in[\Lambda_T,R]\,.
    \end{array}
\right.
$$ 
\item A process $(X_r)_{0\leq r \leq R}$ is called {\it $\Lambda$-adapted} if $X$ is constant on $[\Lambda_{t_-}, \Lambda_t]$, for any $t \in [0,T]$. 
Notice that the terminology {\it $\Lambda$-continuous} is used in \cite{revuz2013continuous} for the same concept. \label{adapted-to-time-change}
\end{enumerate}
\end{definition}
The processes $\Lambda$ and $\Gamma$ as introduced in Definitions \ref{def:time-change} play symmetric roles and we have the following properties:
\begin{enumerate}[label=\textbf{P.\arabic*}]
\item $(\Gamma_r)_{0\leq r\leq R}$ is an increasing right-continuous family of $\FF$-stopping times and the time-changed filtration $(\F_{\Gamma_r})_{0\leq r\leq R}$ given by
$$\F_{\Gamma_r}=\left\{A\in\F_{T}: A\cap\{\Gamma_r\leq u\}\in\F_u \,,\,\forall u\in[0, T]\right\}$$
is a right-continuous filtration (see, e.g.,~\cite[Chapter V, Proposition 1.1]{revuz2013continuous}). \label{P1}
\item $\Gamma$ is $\mathbb{F}_\Gamma$-adapted. Indeed as $\Gamma$ is a family of non-negative $\FF$-stopping times (see \cite[Proposition I.1.28 a)]{Jacod2003} and the right-continuity of the filtration $\FF$), it holds for $t, t'\geq 0$,
$$\{\Gamma_r \leq t'\} \cap \{\Gamma_r \leq t\} = \{\Gamma_r \leq t' \wedge t\} \in \F_{t'\wedge t} \subset \F_t\,,$$
which implies that $\{\Gamma_r \leq t'\} \in \F_{\Gamma_r}$ for all $t'\in [0,T]$, $r\in [0,R]$. \label{P2}
\item \label{P3} It holds $$
\Lambda_t= \left\{
    \begin{array}{ll}
        \inf\{r;\, \Gamma_r>t\}\,, & \mbox{if } t \in [0,T)\,, \\
        \Lambda_T\,, & \mbox{if } t=T\,.
    \end{array}
\right.
$$ 
Since $\Gamma$ is $\mathbb{F}_\Gamma$-adapted, then by symmetry, we deduce that $\Lambda$ is an increasing family of $\FF_\Gamma$-stopping times and the time-changed filtration $(\F_{\Gamma_{\Lambda_r}})_{0\leq r\leq R}$ is a right continuous filtration. 
\item $\Gamma$ is continuous if and only if $\Lambda$ is {\it strictly} increasing. 
In this case, we have
\begin{align}\label{general-inverse}
 \Gamma\circ\Lambda_t=\inf\{s\in[0,T]:\Lambda_s>\Lambda_{t-}\}=t, \qquad 0\leq t\leq T\,,
\end{align}
because either $\Lambda$ is continuous in $t$, in which case $\Lambda_{t-}=\Lambda_t$ or $\Lambda$ jumps in $t$ in which case $\Lambda_t>\Lambda_{t-}$. 
{\it But} notice that if $\Gamma$ is continuous, then $\Lambda$ is still {\it only right-continuous} in general. \label{P4}
\item When $\Lambda$ is strictly increasing, we have $\F_{\Gamma_{\Lambda_t}} =\F_t$\,, for all $t \in [0,T]$, see, e.g., \cite[Proposition 9.9 (iii)]{kallenberg2006foundations}. \label{P5}
\end{enumerate}
Observe that we do not necessarily have $\F_{\Lambda_{\Gamma_r}} = \F_r$, for $r \in [0,R]$ as $\Gamma$ is only increasing and has a flat part due to the fact that $\Lambda$ admits jumps.

\medskip
Our work features controlled dynamics where $M$ is part of the driving noise. Hence we shall consider stochastic integration with respect to $M$. In this context we will work with change of variable formulas for integrals with respect to $M$. In the literature we can find two results in this direction, see {\cite[Theorem 3.1]{kobayashi2011stochastic}} and \cite[Proposition 10.21]{jacod2006calcul} respectively reported in the following two statements in Lemma \ref{lem:change-of-variable} below. For these results, observe that if $(S_r)_{0\leq r\leq R}$ is an $\HH$-semimartingale, then the time-changed process $S_{\Lambda}$ is also an $\hat{\HH}$-semimartingale (see \cite[Corollary 10.12]{jacod2006calcul}). Also for a given 
$\HH$-semimartingale $(S_r)_{0\leq r \leq R}$, we will denote the class of $S$-integrable processes by $L(S, \mathbb{H})$. That is the class of $\HH$-predictable processes for which the It\^o stochastic integral with respect to $S$ is well defined.
\begin{lemma}\label{lem:change-of-variable}
	For a general time-change $\Lambda$, let $S$ be a $\Lambda$-adapted $\HH$-semimartingale and $\Gamma$ be its generalised inverse (see Definition \ref{def:time-change}). Then it holds
	\begin{itemize}
\item[i)]
	 If $\nu \in L(S_\Lambda, \hat{\mathbb{H}})$, then $\nu_{\Gamma_-} \in L(S, \hat{\mathbb{H}}_{\Gamma})$ and 
	\begin{equation*}\label{eq:change-of-variable1}
	\int\limits_0^{t}\nu_s \, \dS_{\Lambda_s} = \int\limits_0^{\Lambda_t}\nu_{\Gamma_{s-}}\, \dS_s\,,  \quad \mbox{a.s.} \quad \forall \,0\leq t\leq T\,.
	\end{equation*}
\item[ii)] 
If $\tilde{\nu} \in L(S, \mathbb{H})$, then $\tilde{\nu}_{\Lambda_{-}} \in L(S_{\Lambda}, \hat{\mathbb{H}})$ and 
	\begin{equation}\label{eq:change-of-variable2}
	\int\limits_0^{\Lambda_t}\tilde{\nu}_s \, \dS_s = \int\limits_0^t\tilde{\nu}_{\Lambda_{s-}}\, \dS_{\Lambda_s}, \quad \mbox{a.s.} \quad \forall \,0\leq t\leq T\,.
	\end{equation}
	\end{itemize}
\end{lemma}
Observe that $\hat{\mathbb{H}}_{\Gamma}$ is well defined because $\Gamma$ is a family of $\hat{\mathbb{H}}$-stopping times by Property \ref{P1} and \eqref{eq:F-in-H-hat}.

In our context, if the time-change $\Lambda$ was continuous, then $S=W$ would be trivially $\Lambda$-adapted, thus the results above would apply. However, we aim at working with a {\it general} time-change (see Definition \ref{def:time-change}) and then $W$ is not necessarily $\Lambda$-adapted. We resolve introducing a new change of variable formula that suits our purposes.	
	
In the next section, we prove the change of variable formulas 
	for integrals w.r.t.~a time-changed Brownian motion, where \textit{general} time-changes are considered.

\section{Change of variable formulas for integrals w.r.t~a time-changed Brownian motion}\label{S4}

Let $(D_{[0,T]}, d)$ be the Skorohod space of c\`adl\`ag real-valued functions on $[0,T]$, see \cite{billingsley2013} for an introduction to Skorohod spaces. 
Define the set $\mathbb{S}$ as 
\begin{equation}\label{eq:setS}
\mathbb{S}= \{\Lambda \mid \Lambda: \Omega \times [0,T]\rightarrow [0,R]\,\, \text{increasing}, \Lambda(\omega, \cdot) \in  D_{[0,T]} \,\, \text{with finitely many jumps,}\,\, \Lambda_0 = 0\}.
\end{equation}

The aim in this section is to write the stochastic integral of $\nu$ w.r.t.~$M$ as a stochastic integral of $\nu$ w.r.t.~$W$ with $\nu$ being a process satisfying the following general condition.


\begin{assumption}\label{assumption-2} 
Let $\Lambda_{[0,s)}(u) = \Lambda_u \mathbf{1}_{[0,s)}(u)$, $u\in [0,T]$ and define $M_{[0,s)}$ similarly.
 Assume $\nu:[0,T]\times \Omega \rightarrow \R$ is a functional of the past of $\Lambda$ and the past of $M$ as follows
 \begin{align*}
  \nu_s=\bar\nu\left(\Lambda_{[0,s)},M_{[0,s)},s\right)\,,
 \end{align*}
for a continuous functional $\bar{\nu}: (\mathbb{S}, \mathcal{B}(\mathbb{S})) \times (D_{[0,T]}, \mathcal{B}(D_{[0,T]})) , \times([0,T], \mathcal{B}([0,T])) \rightarrow \R$, where $\mathcal{B}(\cdot)$ denotes the Borel sigma-algebra of a given set. By continuity, we mean that for all $(x,y, s) \in \mathbb{S} \times D_{[0,T]} \times [0,T]$, for all $\varepsilon>0$, there exists $\delta>0$, such that 
$$|\bar{\nu}(\bar{x}, \bar{y},\bar{s})-\bar{\nu}(x,y,s)|<\varepsilon, \quad \mbox{for all } \bar{x}, \bar{y}, \bar{s} \mbox{ satisfying } \quad \sup\{d(\bar{x}, x), d(\bar{y},y), |\bar{s}-s|\} < \delta\,.$$
\end{assumption}
\medskip

First let us consider the special situation where $\Lambda$ is deterministic. To avoid misunderstandings, we write $\lambda$ instead of $\Lambda$ and $\gamma$ instead of $\Gamma$. Notice that in this case, the filtration 
$\Hh_t = \F^W_t$, $t\in [0,R]$ and $\Hh_{\lambda_t} =\F^W_{\lambda_t} = \sigma\{W_s, \,  s\leq \lambda_t\}$, $t\in [0,T]$. The latter follows from \cite[Chapter 1, Theorem 6]{shiryaev2007optimal}), where it is shown that under some conditions on the probability space the stopped filtration is the filtration generated by the stopped process. Moreover, recall that $W_\lambda$ is an $\FF^W_{\lambda}$-martingale. 
Then we have the following lemma from \cite[Lemma 2.2]{kusuoka2010malliavin}. 
\begin{lemma}\label{DeterministicTrafoformula}
Let $\lambda: [0,T] \rightarrow [0,R]$ be a right-continuous increasing deterministic function that has only finitely many points of discontinuity and is such that $\lambda_0=0$. Let $\gamma$ be the inverse function of $\lambda$. Define 
$(\F^W_{\lambda_t})_{0\leq t\leq T}$. Let $\nu$ be an $(\F^W_{\lambda_t})_{0\leq t\leq T}$-adapted c\`adl\`ag process and $\nu_-$ its left-limit process. Assume $\nu$ satisfies 
 \begin{align*}
  \EE\left[\int\limits_0^T|\nu_{s-}|^2\,\dlambda_s\right]<\infty\,.
 \end{align*}
Then it holds that $\nu_- \circ \gamma$ is $\FF^W$-adapted and  
\begin{align}\label{DeterministicTrafoformulaEq}
 \int\limits_0^t\nu_{s-}\,\dW_{\lambda_s}=\int\limits_0^{\lambda_t}\nu_-\circ \gamma_s\,\dW_s\,, \quad \text{a.s.} \quad t \in [0,T]\,.
\end{align}
The integral in the left-hand side of \eqref{DeterministicTrafoformulaEq} is in the sense of stochastic integrals by $\FF^W_\lambda$-martingales and that of the right-hand side is in the sense of stochastic integrals by $\FF^W$-martingales.
\end{lemma}
As a consequence of Assumption \ref{assumption-2}, we have the following properties of $\nu$.



\begin{lemma}\label{lem:measurability}
 Let $\FF$ be as in \eqref{filtration-M-lambda} and $\nu$ satisfy Assumption \ref{assumption-2}. Then 
 \begin{enumerate}
 \item $\nu$ is left-continuous. Namely, for all $s \in [0,T]$, $\nu_s= \lim_{s_n \uparrow s}\nu(s_n)$\,. \label{left-continuity}
  \item $(\nu_{s})_{s\in[0,T]}$ is $\FF$-adapted. \label{F-adapted}
  \item Let $\lambda$ be as in Lemma \ref{DeterministicTrafoformula}. Then $\left(\bar\nu(\lambda_{[0,t)},W_{\lambda_{[0,t)}},t)\right)_{t\in[0,T]}$ is $(\F^W_{\lambda_t})_{t\in[0,T]}$-adapted. \label{F-W-adapted}
 \end{enumerate}
\end{lemma}

\begin{proof}
Let $\Lambda^-_{[0,s)}(t) := \Lambda(t-) \mathbf{1}_{[0,s)}(t)$. Consider a sequence $(s_n)_{n\geq 0}$ with $s_n \uparrow s$ and $(\theta_n)_{n\geq 0}$ with  $\theta_n(t)=t$, for all $n$. It holds $\Lambda^-_{[0,s_n)}(\theta_n(t)) = \Lambda^-_{[0,s_n)}(t) = \Lambda(t-) \mathbf{1}_{[0,s_n)}(t)$.
Hence 
$$|\Lambda^-_{[0,s_n)}(\theta_n(t)) -\Lambda^-_{[0,s)}(t) |=\Lambda(t-) \mathbf{1}_{[s_n,s)(t)} \leq \Lambda(T) \mathbf{1}_{[s_n,s)(t)}\,,$$
which goes to $0$ when $n$ goes to $\infty$. It follows that $\lim_{n \rightarrow \infty} \Lambda^-_{[0,s_n)} = \Lambda^-_{[0,s)}$ with respect to the Skorohod topology.
Similarly, we prove that $\lim_{n \rightarrow \infty} M^-_{[0,s_n)} = M^-_{[0,s)}$ with respect to the Skorohod topology and 
statement \ref{left-continuity} follows in view of the hypothesis of continuity on $\bar{\nu}$.

To prove \ref{F-adapted} and \ref{F-W-adapted}, it is enough to observe that $\omega\mapsto(\Lambda_u(\omega)\ind_{[0,s)}(u))_{u\in[0,T]}$ is $\mathcal{B}(\mathbb{S})$-$\F^\Lambda_s$-measurable, that 
 $\omega\mapsto(M_u(\omega)\ind_{[0,s)}(u))_{u\in[0,T]}$ is $\mathcal{B}(D([0,T]))$-$\F^M_s$-measurable, for all $s \in [0,T]$, and that the functional $\bar\nu$ is continuous hence measurable, to conclude.
 \end{proof}

In the following theorem, we prove a change of variable formula of the type \eqref{DeterministicTrafoformulaEq} for $\Lambda$ stochastic.
\begin{theorem}\label{TrafoFormula}
 Let $\Lambda$ and $W$ be as in \eqref{W-Lambda} and $\F^\Lambda_T$ and $\F^W_{R}$ satisfy \eqref{Assumption 1}. Assume
 $\Lambda\in\mathbb{S}$ a.s., where $\mathbb{S}$ is as in \eqref{eq:setS}. Moreover, assume $\nu \in L(M, \mathbb{F})$ satisfies Assumption \ref{assumption-2}, and
 \begin{align}\label{finite-second-moment1}
 \EE\left[\int\limits_0^T|\nu_{s}|^2\,\text{d}\Lambda_s\right]<\infty\,.
 \end{align}
 Then it holds $\nu \circ \Gamma \in L(W, \mathbb{H})$ and 
 \begin{align}\label{TrafoFormulaEq}
  \int\limits_0^t\nu_{s}\,\dM_s=\int\limits_0^{\Lambda_t}\nu(\Gamma_s)\,\dW_s=:\int\limits_0^{R}\ind_{[0,\Lambda_t)}(s)\nu(\Gamma_s)\,\dW_s, \quad \text{a.s.} \quad \forall \,0\leq t\leq T\,.
 \end{align}
\end{theorem}

\begin{proof}
By 
Proposition \ref{M-F-martingale} and Lemma \ref{lem:measurability} \ref{F-adapted}, we know that the left hand-side of \eqref{TrafoFormulaEq} makes sense. We show that the right-hand side is well-defined too. For this purpose, we show that the integrand is $\HH$-adapted as $W$ is an $\HH$-Brownian motion (Proposition \ref{W-BM}).
As $\Lambda_t$ is $\Hh_0$-measurable, so is $\ind_{[0,\Lambda_t)}(s)$. It remains to show that $\nu(\Gamma_s)$ is $\Hh_s$-measurable for all $s\in[0,R]$. We have
 \begin{align*}
  \nu(\Gamma_s)=\bar\nu\left(\left(\Lambda_u\ind_{[0,\Gamma_s)}(u)\right)_{u\in[0,T]},\left(W_{\Lambda_u}\ind_{[0,\Gamma_s)}(u)\right)_{u\in[0,T]},\Gamma_s\right)\,.
 \end{align*}
We know that $\Lambda_u\ind_{[0,\Gamma_s)}(u)$, for every $u \in [0,T]$ and $\Gamma_s$ are $\Hh_0$-measurable. It suffices to show that $W_{\Lambda_u}\ind_{[0,\Gamma_s)}(u)$ is $\Hh_s$-measurable for every $u\in [0,T]$. Let $u<\Gamma_s$. Then there exists an $\varepsilon>0$ such that $u<\Gamma_s-\varepsilon$. Therefore, by the monotonicity of $\Lambda$,
 \begin{align*}
  \Lambda_u \leq  \Lambda_{\Gamma_s-\varepsilon}\leq  \lim_{\delta\downarrow 0}\Lambda_{\Gamma_s-\delta}=:\Lambda_{\Gamma_s-}\leq s\,.
 \end{align*}
The latter implies that $\Hh_{\Lambda_u}\subseteq\Hh_s$\,.
 We thus have that $W_{\Lambda_u}$ is $\Hh_s$-measurable for all $u<\Gamma_s$. It follows that $W_{\Lambda_u}\ind_{[0,\Gamma_s)}(u)$ is $\Hh_s$-measurable for every $u\in[0,T]$. This shows that the integrand at the right-hand side of \eqref{TrafoFormulaEq} is $\HH$-adapted.
 
 As both sides of equation \eqref{TrafoFormulaEq} are random variables on $(\Omega,\F,\mathbb{P})$, we need to show that, for each $H\in\F$,
 \begin{align}\label{equality-expectation}
  \EE\left[\int\limits_0^t\nu_s\,\dM_s\ind_H\right]=\EE\left[\int\limits_0^{R}\ind_{[0,\Lambda_t)}(s)\nu(\Gamma_s)\,\dW_s\ind_H\right],
 \end{align}
because we then have $\int_0^t\nu_s\dM_s=\EE\left[\int_0^t\nu_s\,\dM_s\mid \F\right]=\int_0^{R}\ind_{[[0,\Lambda_t))}(s)\nu(\Gamma_s)\dW_s$. 
 But as $\{A_\Lambda\times B_W, A_\Lambda\subset\Omega_\Lambda, B_W\subset\Omega_W\}\cup\N$ is a $\pi$-system that includes an exhausting sequence for $\Omega$ and generates $\F$, by \cite[Theorem 23.9]{schilling2017measures} it actually suffices to show \eqref{equality-expectation} for all $H\in\{A_\Lambda\times B_W, A_\Lambda\subset\Omega_\Lambda, B_W\subset\Omega_W\}\cup\N$. For $H\in\N$, \eqref{equality-expectation} clearly holds (both sides equal $0$), hence w.l.o.g., we consider $H=A_\Lambda\times B_W$ for some $A_\Lambda\subset\Omega_\Lambda$, $B_W\subset\Omega_W$. 
 Define $\tilde\Lambda:[0,T]\times\Omega_\Lambda\rightarrow[0,R]$ by $\tilde\Lambda_t(\omega_\Lambda)=\Lambda_t(\omega_\Lambda)$ and $\tilde W:[0,R]\times\Omega_W\rightarrow\IR$ by $\tilde W_t(\omega_W)=W_t(\omega_W)$.
 Then
 \begin{align}\label{eq:change-of-variable}
  &\EE\left[\int\limits_0^t\nu_s\,\dM_s\ind_H\right]
  \nonumber\\
  &=\int\limits_{A_\Lambda}\int\limits_{B_W}\int\limits_0^t\bar\nu(\Lambda_{[0,s)},M_{[0,s)},s)\,\dM_s\,\text{d}\PP_W\,\text{d}\PP_\Lambda\nonumber\\
  &=\int\limits_{A_\Lambda}\int\limits_{B_W}\int\limits_0^t\bar\nu(\tilde\Lambda(\omega_\Lambda)_{[0,s)},\tilde W_{\tilde\Lambda(\omega_\Lambda)_{[0,s)}}(\omega_W),s)\,\text{d}\tilde W_{\tilde\Lambda_s(\omega_\Lambda)}(\omega_W)\PP_W(\domega_W)\PP_\Lambda(\domega_\Lambda)\nonumber\\
  &=\int\limits_{\tilde\Lambda(A_\Lambda)}\int\limits_{B_W}\int\limits_0^t\bar\nu(\lambda_{[0,s)},\tilde W_{\lambda_{[0,s)}}(\omega_W),s)\,\text{d}\tilde W_{\lambda_s}(\omega_W)\PP_W(\domega_W)(\PP_\Lambda\circ\Lambda^{-1})(\dlambda),
 \end{align}
 where we used a change of variable formula for Lebesgue-integrals. Using similar computations, we get 
 \begin{align*}
& \EE\left[\int_0^T|\nu_s|^2\,\text{d}\Lambda_s\right] \\
&=\int\limits_{\Omega_\Lambda}\int\limits_{\Omega_W}\int\limits_0^T\left|\bar\nu(\Lambda_{[0,s)},M_{[0,s)},s)\right|^2\,\text{d}\Lambda_s\,\text{d}\PP_W\,\text{d}\PP_\Lambda\\
&= \int\limits_{\tilde{\Lambda}(\Omega_\Lambda)} \EE_{\mathbb{P}_W} \left[\int\limits_0^T \left|\bar\nu(\lambda_{[0,s)},\tilde W_{\lambda_{[0,s)}},s)\right|^2\, \dlambda_s\right] (\PP_\Lambda\circ\Lambda^{-1})(\dlambda)\, .
 \end{align*}
 From \eqref{finite-second-moment1}, it follows that $\EE_{\mathbb{P}_W} [\int_0^T |\bar\nu(\lambda_{[0,s)},\tilde W_{\lambda_{[0,s)}},s)|^2\, \dlambda_s]  <\infty$.
Moreover, from Lemma \ref{lem:measurability}, we know that $\left(\bar\nu(\lambda_{[0,t)},\tilde{W}_{\lambda_{[0,t)}},t)\right)_{t\in[0,T]}$ is $(\F^W_{\lambda_t})_{t\in[0,T]}$-adapted.
Hence applying Lemma 
 \ref{DeterministicTrafoformula} on the inner integral in \eqref{eq:change-of-variable} and the change of variable formula for Lebesgue-integrals, we get
 \begin{align*}
  &\EE\left[\int\limits_0^t\nu_s\dM_s\ind_H\right]\\
  &=\int\limits_{\tilde\Lambda(A_\Lambda)}\int\limits_{B_W}\int\limits_0^{\lambda_t}\bar\nu(\lambda_{[0,\gamma_s)},\tilde W_{\lambda_{[0,\gamma_s)}}(\omega_W),\gamma_s)\,\text{d}\tilde W_{s}(\omega_W)\PP_W(\domega_W)(\PP_\Lambda\circ \Gamma)(\dlambda)\\
  &=\int\limits_{A_\Lambda}\int\limits_{B_W}\int\limits_0^{\Lambda_t}\bar\nu(\Lambda_{[0,\Gamma_s)},W_{\Lambda_{[0,\Gamma_s)}},\Gamma_s)\,\text{d}W_{s}\,\text{d}\PP_W\,\text{d}\PP_\Lambda\\
  &=\EE\left[\int\limits_0^{R}\ind_{[0,\Lambda_t)}(s)\nu(\Gamma_s)\,\dW_s\ind_H\right]\,
 \end{align*}
 and the statement follows.
\end{proof}

%
%
%
The results we presented in the latter theorem are written in general terms in the sense that $\nu$ depends on the whole path of $\Lambda$ and $M$. We remark that these results also hold in case $\nu$ depends only on $\Lambda_u$ or $M_u$ at $u \in [0,T]$.

In the following theorem, we write the time-change stochastic integral w.r.t.~the Brownian motion $W$ in terms of the stochastic integral w.r.t~the time-changed Brownian motion. This is a delicate procedure, as it fails without further conditions.
\begin{theorem}\label{TrafoFormulaBack}\hspace{10cm}\\
 Let $\Lambda \in \mathbb{S}$, a.s., and $\tilde\nu \in L(W, \mathbb{H})$. Assume $\tilde{\nu}$ is $\Lambda$-adapted in the sense of Definition \ref{def:time-change} \ref{adapted-to-time-change}. Then $\tilde\nu\circ\Lambda \in L(M, \hat{\mathbb{H}})$ and it holds
 \begin{align}\label{TrafoFormulaBackEq}
\int\limits_0^{\Lambda_t}\tilde\nu_{s}\,\dW_s=\int\limits_0^t\tilde\nu\circ\Lambda_{s}\,\dM_s, \quad \text{a.s.} \quad \forall \, 0\leq t\leq T\,.
 \end{align}
\end{theorem}

 \begin{proof}
  Define $\tau_0:=0$ and let $\tau_i$, $i=1,2,\ldots, N$ be the jump times of $\Lambda$. Observe that these are all $\HH$-stopping times. Therefore, we can write
  \begin{align}\label{eq:sum}
   \int\limits_0^{\Lambda_t}\tilde\nu_{s}\,\dW_s=\sum_{i=1}^{N}\left(\int\limits_{\Lambda_{\tau_{i-1}\wedge t}}^{\Lambda_{(\tau_{i}\wedge t)-}}\tilde\nu_s\,\dW_s+\int\limits_{\Lambda_{(\tau_{i}\wedge t)-}}^{\Lambda_{\tau_{i}\wedge t}}\tilde\nu_s\,\dW_s\right)\,.
  \end{align}
 Now we consider the two terms on the right hand-side of \eqref{eq:sum} separately. As $\Lambda$ is continuous on $[\tau_{i-1}\wedge t,\tau_{i}\wedge t)$, 
 then applying Lemma \ref{lem:change-of-variable}, equation \eqref{eq:change-of-variable2}, we deduce
  \begin{align}\label{lambda-continuous}
  \int\limits_{\Lambda_{\tau_{i-1}\wedge t}}^{\Lambda_{(\tau_{i}\wedge t)-}}\tilde\nu_s\,\dW_s=\int\limits_{(\tau_{i-1}\wedge t,\tau_{i}\wedge t)}\tilde\nu\circ\Lambda_s\,\dW_{\Lambda_s}=\int\limits_{(\tau_{i-1}\wedge t,\tau_{i}\wedge t)}\tilde\nu\circ\Lambda_s\,\dM_s\,.
 \end{align}
 On the other hand, because $\tilde\nu$ is constant on $[\Lambda_{(\tau_{i}\wedge t)-},\Lambda_{\tau_{i}\wedge t}]\subseteq[\Lambda_{\tau_{i}-},\Lambda_{\tau_{i}}]$, it holds that
 $$\tilde\nu_s\equiv\tilde\nu\circ\Lambda_{(\tau_{i}\wedge t)-}\equiv\tilde\nu\circ\Lambda_{\tau_{i}\wedge t}\,.$$ 
 The latter is $\Hh_{\Lambda_{(\tau_{i}\wedge t)-}}$-measurable. Hence $\tilde{\nu}$ in the second integral term in the right-hand side of \eqref{eq:sum} can be pulled out of the integral and we get
 \begin{align}\label{nu-lambda-adpated}
  \int\limits_{\Lambda_{(\tau_{i}\wedge t)-}}^{\Lambda_{\tau_{i}\wedge t}}\tilde\nu_s\,\dW_s&=\tilde\nu\circ\Lambda_{(\tau_{i}\wedge t)-}\cdot\int\limits_{\Lambda_{(\tau_{i}\wedge t)-}}  ^{\Lambda_{\tau_{i}\wedge t}}\dW_s\nonumber\\
  &=\tilde\nu\circ\Lambda_{(\tau_{i}\wedge t)-}\cdot\left(W_{\Lambda_{\tau_{i}\wedge t}}-W_{\Lambda_{(\tau_{i}\wedge t)-}}\right)\nonumber\\
  &=\tilde\nu\circ\Lambda_{\tau_{i}\wedge t}\cdot\left(M_{\tau_{i}\wedge t}-M_{(\tau_{i}\wedge t)-}\right)\nonumber\\
  &=\tilde\nu\circ\Lambda_{\tau_{i}\wedge t}\cdot\int\limits_{[\tau_{i}\wedge t]}\dM_s\nonumber\\
  &=\int\limits_{[\tau_{i}\wedge t]}\tilde\nu\circ\Lambda_s\,\dM_s\,.
 \end{align}
 Summing up \eqref{lambda-continuous} and \eqref{nu-lambda-adpated} yields the statement of the theorem.
 \end{proof}

Note that we assume that $\Lambda$ has only finitely many jumps as we use Lemma 3.1 from  \cite[Lemma 2.2]{kusuoka2010malliavin} where this assumption is needed.

\section{Application to a utility maximisation problem}\label{section:optimal-control}\label{S5}
In this section we aim at applying our change of variable formulas to solve a utility maximisation problem from terminal wealth where the time-change is modelled by a \textit{strictly} increasing process $\Lambda\in \mathbb{S}$, \eqref{eq:setS}. Let us introduce the subset
\begin{equation}\label{eq:setS+}
\mathbb{S}^+=\{\Lambda \mid \Lambda \in \mathbb{S} \mbox{ and strictly increasing}\}.
\end{equation}
Recall that in this case the process $\Gamma$ in Definition \ref{def:time-change} is continuous.
\subsection{The optimisation problem}
We consider a market model that consists of a bond paying zero interest rate and a stock whose value process is given by the $\FF$-semimartingale $S$ with the decomposition 
\begin{align}\label{semimartingale}
S_t=S_0+M_t+A_t\,, \qquad 0 \leq t\leq \bar{T}\,, 
\end{align}
where $S_0$ is a constant, $M$ is as in \eqref{time-changed-bm} and $A$ is such that $A \circ \Gamma$ is $\mathbb{H}$-predictable.
We assume there exists a probability measure $\mathbb{Q}\sim\PP$ such that $S$ is a local $\mathbb{Q}$-martingale. 
We define the space $\Theta$ by
 $$\Theta:= \left\{ \theta \in L(S, \mathbb{F})\,  \left|\, \EE\left[\int\limits_0^T\theta^2_{s-}\, \mathrm{d} \Lambda_s\right] <\infty \right\}\right.\,.$$
\par
A self-financing strategy $\nu \in \Theta$ starting at time $t$ with the starting value $x\geq 0$ has at time $t_1$\,, the value
\begin{align*}
 V_{t_1}^{t,x}(\nu)=x+\int\limits_t^{t_1}\nu_{u-}\, \dS_u\,, \qquad  0\leq t\leq t_1 \leq T\,.
\end{align*}
The component $\nu$ of the trading strategy corresponds to the amount of money invested in the asset $S$.
 The set of admissible strategies that we want to allow for shall be given in the following definition.
\begin{definition}[\textbf{admissible trading strategies $\A_{t;\FF}$}]\label{AdmissibleStrategies}
 The set of admissible trading strategies $\A_{t;\FF}$ consists of all processes $(\nu_s)_{t\leq s\leq T}$ fulfilling Assumption \ref{assumption-2}, and such that
 \begin{enumerate}
\item $\nu \in \Theta$,
 \item  the strategy is such that the discounted wealth process 
 \begin{equation}\label{value-process}
 V_T^{t,x}(\nu) = x+ \int\limits_t^T \nu_u \, \ud S_u\,
 \end{equation}
 is non-negative.
  \end{enumerate}

\end{definition}
%
The goal is to find an admissible strategy $\nu^* \in \A_{t;\FF}$ under which the conditional expected utility of the terminal wealth 
\begin{align*}
 J^{t,x}(\nu):=\EE\left[U\left(V_T^{t,x}(\nu)\right)  \mid  \F_t\right]\,,
\end{align*}
is maximised for $t \in [0,T]$\,. Thus we want to find $\nu^*$ such that
\begin{align}\label{OptimizationProblem}
 J^{t,x}(\nu^*)=\esssup_{\nu\in\A_{t;\FF}}J^{t,x}(\nu)\,, \quad t \in [0,T]\,.
\end{align}

\par \medskip
The problem of maximising expected utility from terminal wealth is a classical problem in mathematical finance (we refer, e.g.,~to \cite{karatzas1998methods} for an overview). Different approaches are used in the literature to solve such a problem. One approach based on the the theory of partial differential equations is studied, e.g.,~in \cite{merton1969lifetime, merton1975optimum, benth2003merton, framstad2001optimal} in a Markovian setting. Other approaches based on duality characterisations of portfolios or the theory of quadratic backward stochastic differential equations
 are considered, e.g.,~in \cite{cox1991variational, karatzas1987optimal, karatzas1991martingale, he1991consumption, kramkov1999asymptotic, hu2005utility, morlais2009quadratic, morlais2010new, Oksendal-Sulem} in a continuous and jump setting. 
The case where the price process is modelled by a time-changed L\'evy process with the time-change being absolutely continuous w.r.t.~the Lebesgue-measure is considered in \cite{di2014bsdes, kallsen2010utility}.
 
 \par \medskip
Hereafter, we tackle the problem \eqref{OptimizationProblem} for price processes modelled by a semimartingale $S$ whose decomposition is as described in \eqref{semimartingale}. 
Our approach is to first take the conditioning on the sigma-algebra $\Hat{\Hh}_t$ introduced in \eqref{TCFhat}-\eqref{TCFLambda} 
and then use the change of variable formula in Theorem \ref{TrafoFormulaEq} in order to translate the integral w.r.t.~the martingale $M$ into an integral w.r.t.~the Brownian motion $W$ and solve the problem in this setup. Afterwards we will relate the solution under the enlarged filtration to the solution under the original one. 



 \subsection{The optimisation problem under the enlarged filtration}
 Instead of optimising under the filtration $\FF$, let us first suppose we are given the information in $(\hat{\Hh}_{_t})_{0 \leq t \leq T}$. Then the optimisation objective becomes 
\begin{align}\label{EnlargedObjective}
 J_{\HH}^{t,x}(\nu):=\mathbb{E}\left[U(V_T^{t,x}(\nu))\mid \Hh_{\Lambda_t}\right]\,, \qquad t\in [0,T]\,,
\end{align}
i.e., we want to find $\hat\nu \in\A_{t;\FF}$ such that
\begin{align}\label{EnlargedOptimizationProblem}
 J_{\HH}^{t,x}(\hat\nu)=\esssup_{\nu\in\A_{t;\FF}}J_{\HH}^{t,x}(\nu)\,, \qquad t\in [0,T]\,.
\end{align}
Using Theorem \ref{TrafoFormula}, and the change of variable formula for the Lebesgue-measure, we derive
\begin{align}\label{tra-to-continous}
U(V_T^{t,x}(\nu)) &= U\left(x+\int\limits_t^T\nu_u\,\dS_u\right)\nonumber \\
&=U\left(x+\int\limits_t^T\nu_u\,\dM_u+\int\limits_t^T\nu_u \,\ud A_u\ \right)\nonumber \\
&= U\left(x+\int\limits_{\Lambda_t}^{\Lambda_T}\nu_{\Gamma_u}\,\dW_u+\int\limits_{\Lambda_t}^{\Lambda_T}\nu_{\Gamma_u}\,\text{d}(A\circ \Gamma)_u\right).
\end{align}
%

In the sequel we introduce a new set of admissible strategies which will allow us to investigate the optimal problem in the continuous setting of \eqref{tra-to-continous}.
\begin{definition}[\textbf{admissible trading strategies $\tilde{\A}_{\Lambda_t;\HH}$}]
Let $\ud X_r = \ud W_r + \text{d}(A\circ \Gamma)_r$, $0\leq r\leq R$. The set of admissible trading strategies $\tilde{\A}_{\Lambda_t;\HH}$ consists of all  c\`agl\`ad processes $(\tilde{\nu}_r)_{\Lambda_t\leq r\leq R} \in L(X, \mathbb{H})$ such that 
 \begin{enumerate}
\item $\tilde\nu$ is $\Lambda$-adapted,
 \item the discounted wealth process 
 \begin{equation}\label{value-process1}
 V_R^{\Lambda_t,x}(\tilde{\nu}) = x+ \int\limits_{\Lambda_t}^R \tilde{\nu}_u \, \ud X_u\,, \qquad 0\leq t\leq T
 \end{equation}
 is non-negative.
  \end{enumerate}
\end{definition}

\begin{proposition}\label{prop:problem-with-inequality} Let $S$ be an $\mathbb{F}$-semimartingale with decomposition \eqref{semimartingale}. 
Let 
$\Lambda \in \mathbb{S}^+$ a.s., with $\mathbb{S}^+$ given in \eqref{eq:setS+}.  Then, for \eqref{EnlargedOptimizationProblem}, it holds that
\begin{align}
 J_{\HH}^{t,x}(\hat{\nu})
 &\leq  \esssup_{\tilde{\nu}\in\tilde\A_{\Lambda_t;\HH}}\mathbb{E}\left[U\left(x+\int\limits_{\Lambda_t}^{\Lambda_T}\tilde\nu_u\,\dW_u+\int\limits_{\Lambda_t}^{\Lambda_T}\tilde\nu_u \,\ud(A\circ \Gamma)_u\right)\mid \Hh_{\Lambda_t}\right]\label{final-problem2}\,.
\end{align}
\end{proposition}
\begin{proof}
As $\Gamma$ is continuous, it is obvious that $\tilde{\nu} = \nu\circ\Gamma$ is c\`agl\`ad for each $\nu \in \A_{t;\FF}$\,. 
 Moreover, from \eqref{general-inverse}, we know that $\tilde{\nu} \circ \Lambda = \tilde{\nu} \circ \Lambda_- = \nu$. Hence $\tilde{\nu}$ is $\Lambda$-adapted.
Lemma \ref{lem:change-of-variable} and Theorem \ref{TrafoFormula} yield $\tilde{\nu} \in L(X,\mathbb{H})$. The non-negativity of $V_{R}^{\Lambda_t,x}(\tilde{\nu})$ in \eqref{value-process1} follows from the non-negativity of $V_{T}^{t,x}(\nu)$ in \eqref{value-process}. Therefore for $\nu \in \A_{t;\FF}$, we have that $\tilde{\nu} \in \tilde{\A}_{\Lambda_t;\HH}$ and the statement of the proposition follows. 
\end{proof}

Unfortunately, \eqref{final-problem2} does not hold in general with equality, i.e., optimising over $\tilde\A_{\Lambda_t;\HH}$ in the time-changed framework yields an upper boundary for the solution to the original problem. The reason for this is that for $\tilde\nu$ being $\HH$-adapted, in general $\tilde\nu\circ\Lambda_-$ is not $\FF$-adapted, so $\tilde\nu\circ\Lambda_-$ will not {\it in general} be an admissible strategy. This becomes clear when one keeps in mind that the filtration $\HH$ has all information about the whole path of $\Lambda$ from the very beginning, so most of the ``admissible strategies'' in the set $\tilde\A_{\Lambda_t;\HH}$ would have future information. 

The way we proceed is to impose some conditions on the drift of our model. This condition will allow to construct, for some chosen utility functions, a strategy for the right-hand side of \eqref{final-problem2} and for which the equality will hold.

We model the stock price by an $\FF$-semimartingale as in \eqref{semimartingale}
where we impose a special form to its finite variation part.  
\begin{assumption}\label{assum:assumption-A}
Let $\Gamma$ be as in  Definition \ref{def:time-change}. Assume $A$ is a finite-variation process satisfying
$$ (A \circ \Gamma)_r = \int_0^r \tilde{\theta}_u \, \text{d}u\,, \qquad 0\leq r\leq R\,,$$

for $\tilde\theta$ being an integrable c\`agl\`ad $\Lambda$- and $\mathbb{H}$-adapted process.
\end{assumption}
Notice that Assumption \ref{assum:assumption-A} implies that $A\circ \Gamma$ is $\mathbb{H}$-predictable. Indeed as $\tilde\theta$ is $\mathbb{H}$-progressively measurable, it holds that $A\circ \Gamma$ is a progressively measurable and continuous process and hence predictable.

In the sequel we construct for the case of a \textit{power utility} function and a \textit{logarithmic utility} function, 
\begin{equation}\label{power-utility}
U(x) = \begin{cases}\dfrac{x^{1-p}}{1-p}\,, & p \in \R_+ \setminus\{0,1\}\quad\mbox{(power)},\\
\log(x)\, , & p=1\,, x>0\quad \mbox{(log)},
\end{cases}
\end{equation}
a strategy $\tilde{\nu} \in \tilde\A_{\Lambda_t;\HH}$ such that $\tilde\nu\circ\Lambda_- \in \A_{t;\FF}$ and that is optimal for the right-hand side of \eqref{final-problem2}, i.e., $\tilde{\nu}\circ\Lambda_-$ is optimising
\begin{align}\label{eq:final-problem2}
\mathbb{E}\left[U\left(x+\int\limits_{\Lambda_t}^{\Lambda_T}\tilde\nu_u\,\dW_u+\int\limits_{\Lambda_t}^{\Lambda_T}\tilde\nu_u \tilde{\theta}_u\,\text{d}u\right)\mid \Hh_{\Lambda_t}\right]\,.
\end{align}

To derive our strategy in the following theorem, we adapt the approach in \cite[Theorem 3.1]{kallsen2010utility} and in \cite{GOLL2000} to our setting.

\begin{theorem}\label{thm:another-optimal-problem} 
Let $\Lambda \in \mathbb{S}^+$ a.s., 
and $\pi_u = \tilde{\theta}_u/p$, $u\in [0, R]$, for $\tilde{\theta}$ being as in Assumption \ref{assum:assumption-A} and satisfying
	\begin{equation}\label{eq:generatingandmoment}
			\EE\left[ \exp\left(\int_0^u\tilde{\theta}_s^2\, \ud s\right)\right]<\infty\,, \quad  \mbox{\ for all\ } u\in [0,R].
			\end{equation} Denote by $\E(Y)$ the stochastic exponential of a given semimartingale $Y$.
Then 
\begin{equation}\label{breve-nu}
\tilde{\nu}_s = x \pi_s \,\mathcal{E}\left(\int\limits_0^\cdot \pi_u \, \ud X_u\right)_s\,, \qquad \Lambda_t \leq s\leq  R\,,
\end{equation}
is an admissible strategy in $ \tilde{\A}_{\Lambda_t;\HH}$ that  is optimal for \eqref{eq:final-problem2} with value $V_s^{\Lambda_t,x}(\tilde{\nu}) = x\, \mathcal{E}\left(\int_0^\cdot \pi_u \, \ud X_u\right)_s$, $\Lambda_t \leq s\leq R$ when $\tilde{\theta}_r$ is  $\mathcal{H}_0$-measurable for any $r\in [0,R]$.\\
Moreover, the corresponding maximal  expected \textit{power} utility is given by
$$\EE[U(V_ {\Lambda_{T}}^{\Lambda_t,x}(\tilde{\nu}))\mid \hat\Hh_t] = \frac{x^{1-p}}{1-p}\exp\left\{\int\limits_{\Lambda_t}^{\Lambda_T}\frac{1-p}{2p} \tilde{\theta}_u^2 \, \ud u \right\}\,,$$
while for the maximal \textit{logarithmic} utility we get
$$\EE[U(V_ {\Lambda_{T}}^{\Lambda_t,x}(\tilde{\nu}))\mid \hat\Hh_t] = \log (x)+\int\limits_{0}^{\Lambda_t}\tilde{\theta}_u \,\ud W_u +\frac 12\int\limits_{0}^{\Lambda_T} \tilde{\theta}^2_u\, \ud u\, .$$
\end{theorem}

\begin{proof}
We first check that $\tilde{\nu} \in \tilde{\A}_{\Lambda_t;\HH}$.
By  Assumption \ref{assum:assumption-A} it is obvious that $\tilde{\nu}$ as defined in \eqref{eq:generatingandmoment} is c\`agl\`ad. 
In view of the additional assumption \eqref{breve-nu} on $\tilde{\theta}$, 
we deduce that $\pi \in L(X, \mathbb{H})$ and hence $\tilde{\nu} \in L(X, \mathbb{H})$. As $\tilde{\theta}_u$ is assumed to be constant on $[\Lambda_{u-}, \Lambda_u]$, it follows 
that $\tilde{\nu}$ is also $\Lambda$-adapted.
Moreover, we derive
\[
V_s^{\Lambda_t, x} (\tilde{\nu}) =  x + \int\limits _{\Lambda_t}^s \tilde{\nu}_u \, \ud X_u= x \,\mathcal{E}\left(\int\limits_0^\cdot \pi_u \, \ud X_u\right)_s\,,  \qquad \Lambda_t \leq s\leq R\,,
 \]
from which we deduce that $V_s^{\Lambda_t, x}(\hat{\nu})$ is non-negative. 
Observe that the assumption \eqref{eq:generatingandmoment} on $\tilde{\theta}$ implies $\EE\left[\int_0^{R} \tilde{\nu}_u^2\, \ud u\right]< \infty$.
Indeed, define for $u \in [0,R]$
	$$Z_u=\exp\left\{\frac 2 p\int_0^u\tilde{\theta}_s\,\ud W_s-\frac 12\left(\frac 2p\right)^2 \int_0^u\tilde{\theta}^2_s\, \ud s\right\} \quad \mbox{and} \quad
	b=2\frac{2p-1}{p} +\frac 12\left(\frac 2p\right)^2.$$
	Then, conditioning on the  sigma-algebra $\sigma\{\tilde{\theta}_s, s\leq u\}$, we compute
	\begin{align*}
	&\EE\left[\int_0^{R} \tilde{\theta}^2_u\exp\left\{2\int_0^u\frac{\tilde{\theta}_s}{p}\ud W_s+2\int_0^u\frac{2p-1}{p}\tilde{\theta}^2_s \ud s\right\}\ud u\right]\\
	&\qquad  = \int_0^{R}\EE\left[ \tilde{\theta}^2_u\exp\left\{b\int_0^u\tilde{\theta}_s^2\,\ud s\right\}	\EE\left[Z_u\mid\sigma\{\tilde{\theta}_s, s\leq u\}\right]\right] \ud u\\
	&\qquad =  \int_0^{R}\EE\left[ \tilde{\theta}^2_u\exp\left\{b\int_0^u\tilde{\theta}_s^2\,\ud s\right\}\right]\ud u\,.
	\end{align*}
 Therefore $\tilde{\nu} \in \tilde{\A}_{\Lambda_{t;\HH}}$. \\
For the case of \textit{power utility}, let $\psi$ be another admissible strategy in $\tilde\A_{\Lambda_t; \HH}$. Then we can write
$$\psi_s = \eta_s V_s^{\Lambda_t,x}(\psi)\,, \qquad \Lambda_t \leq s\leq R\,,$$
for an $\R$-valued $\HH$-adapted process $\eta$ and 
$$\ud V_s^{\Lambda_t,x}(\psi) = V_s^{\Lambda_t,x}(\psi) \eta_s\, \ud X_s\,, \qquad \Lambda_t \leq s\leq R.$$
Define $L_t = \exp\left\{\int_t^{\Lambda_T} \alpha_u\, \ud u\right\}$, with $\int_0^R| \alpha_u|\, \ud u < +\infty$ by \eqref{eq:generatingandmoment},
where 
$$\alpha_t:= \frac{1-p}{2p} \tilde{\theta}_t^2\,.$$
The process $L/L_0$ is continuous and of finite variation. Hence it is an $\HH$-semimartingale. Applying the It\^o formula to $$F(L/L_0, V^{\Lambda_t,x}(\tilde{\nu}), V^{\Lambda_t,x}(\psi)) = L/L_0 \left(V^{\Lambda_t,x}(\tilde{\nu})\right))^{-p}V^{\Lambda_t,x}(\psi)\,,$$ 
we deduce that the latter is an $\HH$-martingale. Then since $U$ as defined in \eqref{power-utility} is concave, we have
\begin{align*}
U(V_{\Lambda_T}^{\Lambda_t,x} (\psi)) \leq U\left(V_{\Lambda_T}^{\Lambda_t,x}(\tilde{\nu})\right) + U' \left(V_{\Lambda_T}^{\Lambda_t,x}(\tilde{\nu})\right)\left(V_{\Lambda_T}^{\Lambda_t,x} (\psi) -V_{\Lambda_T}^{\Lambda_t,x}(\tilde{\nu})\right)\,,
\end{align*}
for any admissible strategy $\psi$. This implies 
\begin{align*}
\mathbb{E}\left[U\left(V_{\Lambda_T}^{\Lambda_t,x}(\psi)\right) \mid \hat{\Hh}_t\right] &\leq \mathbb{E}\left[U\left(V_{\Lambda_T}^{\Lambda_t,x}(\tilde{\nu})\right)\mid \hat{\Hh}_t\right]\\
&\qquad +  \mathbb{E}\left[ L_{\Lambda_T}\left(V_{\Lambda_T}^{\Lambda_t,x}(\tilde{\nu})\right)^{-p} V_{\Lambda_T}^{\Lambda_t,x} (\psi) - L_{\Lambda_T}\left(V_{\Lambda_T}^{\Lambda_t,x}(\tilde{\nu})\right)^{1-p} \mid \hat{\Hh}_t \right]\\
&= \mathbb{E}\left[ U\left(V_{\Lambda_T}^{\Lambda_t,x}(\tilde{\nu})\right)\mid \hat{\Hh}_t\right]\,,
\end{align*}
where we used the additional assumption of $\mathcal{H}_0$-measurability, the optional sampling theorem and the fact that $L/L_0\left(V^{\Lambda_t,x}(\tilde{\nu})\right)^{-p} V^{\Lambda_t,x}(\psi)$ and 
$L/L_0\left(V^{\Lambda_t,x}(\hat{\nu}\right))^{1-p}$ are $\HH$-martingales with the same value at $t$. Hence the first claim follows.\\
The corresponding maximal expected utility follows from observing that 
\begin{align*}
\EE\left[U(V_{\Lambda_T}^{\Lambda_t,x}(\tilde{\nu}))\mid \hat{\Hh}_t\right] &= \frac{L_0}{1-p} \EE\left[L_{\Lambda_T}/L_0(V_{\Lambda_T}^{\Lambda_t,x}(\tilde{\nu}))^{1-p} \mid \hat{\Hh}_t\right]\\
&= \frac{x^{1-p}}{1-p} \exp\left\{\int\limits_{\Lambda_t}^{\Lambda_T} \alpha_u\, \ud u\right\}\,.
\end{align*}
For the case of \textit{logarithmic utility}, the optimal strategy \eqref{breve-nu} with $p=1$ directly follows from Theorem 3.1 and Example 4.2 in \cite{GOLL2000} for the terminal wealth case. The corresponding maximal utility is obtained from
\begin{align*}
\EE\left[U(V_{\Lambda_T}^{\Lambda_t,x}(\tilde{\nu}))\mid \hat{\Hh}_t\right] &= \EE\left[\log( x \,\mathcal{E}\left(\int\limits_0^\cdot \pi_u \, \ud X_u\right)_{\Lambda_T})\mid \hat{\Hh}_t\right]\\
 &= \log (x)+\EE\left[\int\limits_{0}^{\Lambda_T}\tilde{\theta}_u \,\ud W_u +\frac 12\int\limits_{0}^{\Lambda_T} \tilde{\theta}^2_u\, \ud u \mid \hat{\Hh}_t\right].
\end{align*} 
Under the additional assumption of $\mathcal{H}_0$-measurability and noting that the stochastic integral w.r.t.\ $W$ is an $\mathbb{H}$-martingale, we get the stated result.
\end{proof}

The additional assumption of $\mathcal{H}_0$-measurability will be satisfied for example when 
the random time process $\Gamma$  is the only stochastic driver for $\tilde\theta$. This assumption is similar to the $\mathcal{G}_0$-measurability of the local characteristics in \cite{kallsen2010utility}.

\par \medskip
In order to optimise \eqref{eq:final-problem2} over the set of strategies  $\A_{t;\FF}$ for the power utility case ($p \in \R_+ \setminus\{0,1\}$) and the logarithmic utility case ($p=1$), we impose a stronger condition on the finite-variaton process $A$.
\begin{assumption}\label{assum:assumption-A2}
Let $\Gamma$ be as in Definition \ref{def:time-change} and $\theta$ a process satisfying Assumption \ref{assumption-2}. Define
$$\tilde{\theta}_r= (\theta \circ \Gamma)_r\,, \qquad 0\leq r\leq R\,.$$
Assume $A$ is a finite-variation process such that $A \circ \Gamma$ is absolutely continuous w.r.t. the Lebesgue measure with density process $\tilde{\theta}$, i.e.,
$$(A\circ \Gamma)_r =\int_0^r\tilde{\theta}_u\, \ud u\,, \qquad 0\leq r\leq R\,.$$
\end{assumption}

With this assumptin, we prove the following technical results.
\begin {lemma}\label{lem:measurability1}
Let $A$ satisfy Assumption \ref{assum:assumption-A2}. 
Then 
\begin{enumerate}
\item $A$ is an $\FF$-adapted c\`adl\`ag process,
\item $A\circ \Gamma$ is $\HH$-predictable.
\end{enumerate}
\end{lemma}
\begin{proof}
It holds that 
$$A_t = \int_0^{\Lambda_t} (\theta \circ \Gamma)_s \, \ud s, \qquad 0\leq t\leq T\,.$$
From Lemma \ref{lem:measurability}, we know that $\theta$ is  $\mathbb{F}$-predictable, from which we deduce that $\tilde{\theta}_r = \theta(\Gamma_r)$ is $\mathbb{F}_{\Gamma_r}$-measurable for all $r \in [0,R]$. It follows from \cite[Proposition I.1.23]{Jacod2003} that $(\tilde{\theta}_r \mathbf{1}_{\{\Gamma_r\leq t\}})_{0\leq r\leq R}$ is optional, for all $t\in [0,T]$. Hence it is progressively measurable, which implies that $A$ is $\mathbb{F}$-adapted.
Moreover $A$ is c\`adl\`ag as a composition of a continuous function with a right-continuous increasing function. Hence we proved the first claim of the lemma.

For the second claim, observe that it follows from Theorem \ref{TrafoFormula}, that $\tilde{\theta}$ is $\mathbb{H}$-adapted. Since it is c\`agl\`ad, then it is $\mathbb{H}$-predictable and hence progressively measurable. The latter implies the second claim of the lemma.
\end{proof}

\begin{theorem}\label{hat_nu}
Let $\Lambda \in \mathbb{S}^+$, a.s. Let $S$ be an $\mathbb{F}$-semimartingale with decomposition \eqref{semimartingale} and finite variation part $A$ satisfying Assumption \ref{assum:assumption-A2}. Assume moreover that $\tilde{\theta}$ 
satisfies \eqref{eq:generatingandmoment}.
Then the strategy $\hat{\nu}$ given by
\begin{align}\label{hat-nu}
\hat{\nu}_u 
&= \frac{(\tilde{\theta} \circ \Lambda)_u x}{p}  \exp\left\{ \int\limits_0^u \frac{(\tilde{\theta} \circ \Lambda)_s}{p}\, \ud M_s
 + \int\limits_0^u \frac{2p-1}{2p^2} (\tilde{\theta} \circ \Lambda)_s \, \ud A_s\right\}\,, \quad 0  \leq u\leq T,
\end{align}
belongs to $\A_{t;\FF}$ and is optimal for the right-hand side of \eqref{final-problem2} when $\tilde{\theta}_r$ is  $\mathcal{H}_0$-measurable for any $r\in [0,R]$.
\end{theorem}

\begin{proof}
As $\tilde{\nu}$ in \eqref{breve-nu} is $\Lambda$-adapted, and $\tilde{\nu} \in L(X, \mathbb{H})$, then applying Theorem \ref{TrafoFormulaBack}, yields
\begin{align*}
\hat{\nu}_u &= \tilde{\nu}\circ\Lambda_u\\
& = \frac{(\tilde{\theta}\circ \Lambda)_u}{p}\, x \exp\left\{\int\limits_0^{\Lambda_u} \frac{\tilde{\theta}_s}{p}\, \ud W_{s} + \int\limits_0^{\Lambda_u} \frac{2p-1}{2p^2} \tilde{\theta}_s^2 \, \ud s\right\}\\
&= \frac{(\tilde{\theta} \circ \Lambda)_u}{p}\,x  \exp\left\{ \int\limits_0^u \frac{(\tilde{\theta} \circ \Lambda)_s}{p}\, \ud W_{\Lambda_s} + \int\limits_0^u \frac{2p-1}{2p^2} (\tilde{\theta} \circ \Lambda)_s \, \ud A_s\right\}\,,
\end{align*} 
which is $\mathcal{F}_u$-measurable, for all $u \geq 0$, c\`agl\`ad, and satisfies Assumption \ref{assumption-2}. 

The additional assumption \eqref{eq:generatingandmoment} on $\tilde{\theta}$ implies $\EE\left[\int_0^{R} \tilde{\nu}_u^2\, \ud u\right]< \infty$ and
since $\hat\nu$ is $\FF$-adapted, it holds 
\begin{equation*}
\EE\left[\int\limits_0^{R} \tilde{\nu}_u^2\, \ud u\right] = \EE\left[\int\limits_0^{\Gamma_{R}}\hat\nu_u^2 \, \ud \Lambda_u\right]\\
=\EE\left[\int\limits_0^{T}\hat\nu^2_u \, \ud \Lambda_u\right]< \infty.
\end{equation*}
Finally, observe that the non-negativity of $V_T^{t,x}(\hat{\nu})$ follows from the non-negativity of $V_{R}^{\Lambda_{t,x}}(\tilde{\nu})$. We conclude that $\hat{\nu} \in \A_{t;\FF}$ and by Theorem \ref{thm:another-optimal-problem}
that it is optimal for \eqref{eq:final-problem2}  when $\tilde{\theta}_r$ is  $\mathcal{H}_0$-measurable for any $r\in [0,R]$.
\end{proof}
Notice that the strategy $\tilde{\nu}$ constructed in Theorem \ref{thm:another-optimal-problem} is such that $\tilde{\nu}\circ \Lambda = \hat{\nu} \in \mathcal{A}_{t;\mathbb{F}}$ and hence we observe that under Assumption \ref{assum:assumption-A2}, we have equality in \eqref{final-problem2}.
\subsection{Solution to the original optimisation problem}
In the next theorem we finally provide a solution to the original optimisation problem (\ref{OptimizationProblem}). 
\begin{theorem}\label{Optimality} Let $S$ be an $\mathbb{F}$-semimartingale with decomposition \eqref{semimartingale} and finite variation part $A$ satisfying Assumption \ref{assum:assumption-A2} with $\Lambda \in \mathbb{S}^+$, a.s.
Moreover let $\hat{\nu} \in \A_{t;\FF}$ be an admissible strategy that is optimal for \eqref{eq:final-problem2}. Then 
 it holds that $\hat\nu$ is also the optimal strategy $\nu^*$ for  \eqref{OptimizationProblem} and
 \begin{align}\label{eq:final-solution}
  J^{t,x}(\nu^*)=\mathbb{E}\left[J_{\HH}^{t,x}(\hat\nu)\mid\F_t\right]\,, \qquad t\in [0,T]\,.
 \end{align}
Under the additional assumptions on $\tilde{\theta}$ as in Theorem \ref{hat_nu}, this optimal strategy $\hat{\nu}$ is given in \eqref{hat-nu} and 
the corresponding maximal  expected \textit{power} utility is given by
\begin{equation}\label{eq:sol-power}
J^{t,x}(\nu^*)= \mathbb{E}\left[\frac{x^{1-p}}{1-p}\exp\left\{\int\limits_{\Lambda_t}^{\Lambda_T}\frac{1-p}{2p} \tilde{\theta}_u^2 \, \ud u \right\}\mid\F_t \right]\,,
\end{equation}
while for the maximal \textit{logarithmic} utility we get
\begin{equation}\label{eq:sol-logarithm}
J^{t,x}(\nu^*) = \mathbb{E}\left[\left( \log(x)+\int\limits_{0}^{\Lambda_t}\tilde{\theta}_u \,\ud W_u +\frac 12\int\limits_{0}^{\Lambda_T} \tilde{\theta}^2_u\, \ud u\right) \mid \F_t \right]\,.
\end{equation}
\end{theorem}

\begin{proof}
Recall the set of admissible strategies $\A_{t;\FF}$ in Definition \ref{AdmissibleStrategies}. Let  $\nu\in\A_{t,\FF}$. Then applying the tower property yields
\begin{align*}
 J^{t,x}(\nu)=\mathbb{E}\left[U(V_T^{t,x}(\nu))\mid\F_t\right]=\mathbb{E}\left[\mathbb{E}\left[U(V_T^{t;x}(\nu))\mid\Hh_{\Lambda_t}\right]\mid\F_t\right]=\mathbb{E}[J_{\HH}^{t,x}(\nu)\mid\F_t].
\end{align*}
In particular, $J^{t,x}(\hat\nu)=\mathbb{E}[J_{\HH}^{t,x}(\hat\nu)\mid\F_t]$. Using the latter together with Jensen's inequality yields
\begin{align*}
 J^{t,x}(\nu^*)=\esssup_{\nu\in\A_{t;\FF}}J^{t,x}(\nu) &= \esssup_{\nu\in\A_{t;\FF}}\mathbb{E}\left[J_{\HH}^{t,x}(\nu)\mid\F_t\right]\\
 &\leq \mathbb{E}\left[\esssup_{\nu\in\A_{t;\FF}}J_{\HH}^{t,x}(\nu)\mid\F_t\right]\\
 &=\mathbb{E}\left[J_{\HH}^{t,x}(\hat\nu)\mid\F_t\right]=J^{t,x}(\hat\nu)\,.
\end{align*}
But as $\hat\nu\in\A_{t;\FF}$, we clearly have that
\begin{align*}
 J^{t,x}(\hat\nu)\leq\esssup_{\nu\in\A_{t;\FF}}J^{t,x}(\nu)=J^{t,x}(\nu^*)
\end{align*}
and \eqref{eq:final-solution} follows. Equations \eqref{eq:sol-power} and \eqref{eq:sol-logarithm} follow from Theorem \ref{thm:another-optimal-problem} and equation \eqref{eq:final-solution}. 
\end{proof}

\par \medskip


\paragraph{Conclusion.} We have derived change of variable formulas for stochastic integrals w.r.t.\ a time-changed Brownian motion. We made use of 
these change of variable formulas to solve the problem of maximising expected utility from terminal wealth in a semimartingale setting where the semimartingale is written 
as a sum of a time-changed Brownian motion and a finite variation process. To solve our problem we needed to impose certain conditions on the finite variation process of the considered semimartingale. These conditions allowed us to obtain explicit expressions for the optimal strategy in terms of the finite-variation process where we consider the cases of power and logarithmic utilities. 

When we do not impose extra conditions on the finite-variation part of the price process, we obtained, under the enlarged filtration, an upper bound for the optimisation problem \eqref{final-problem2} and not a representation. That is we do not have equality in \eqref{final-problem2}, but inequality. 
In a future research, one can investigate whether it exists a larger filtration  where we can have a representation of the optimisation problem  with equality in \eqref{final-problem2} without the need to impose extra conditions on the finite-variation part of the price process. Another interesting study would be to investigate the change of variable formulas for stochastic integrals w.r.t.\ a time-changed Poisson or a more general L\'evy process. 

\par\bigskip
{\bf Acknowledgement.}
The research leading to these results is within the project {\it STORM: Stochastics for Time-Space Risk Models} of the Research Council of Norway (nr.\ 274410). The authors also acknowledge the FWO Scientific Research Network ModSimFIE (FWO WOG W001021N) for funding research visits to carry out this research. 


\bibliographystyle{alpha}
\bibliography{references}

\appendix
\section*{Appendix}\label{proofs}

{\bf Proof of Proposition \ref{AufteilungSigmaAlgF}.} \\
i) We have
\begin{align*}
\F^\Lambda_T\vee\F^W_{R}&=\left((\F_\Lambda\otimes\{\emptyset,\Omega_W\})\vee\N\big)\vee\big((\{\emptyset,\Omega_\Lambda\}\otimes\F_W)\vee\N\right)\\
&=\big((\F_\Lambda\otimes\{\emptyset,\Omega_W\})\vee(\{\emptyset,\Omega_\Lambda\}\otimes\F_W)\big)\vee\N\\
&=\sigma\left(\bigcap_{i=1}^nA_i, A_i\in\{A_\Lambda\times\Omega_W, A_\Lambda\in\F_\Lambda\}\cap\{\Omega_\Lambda\times B_W, B_W\in\F_W\}, n\in\IN \right)\vee\N\\
&=\sigma\big(\{A_\Lambda\times B_W, A_\Lambda\in\F_\Lambda, B_W\in\F_W\}\big)\vee\N\\
&=(\F_\Lambda\otimes\F_W)\vee\N 
 \end{align*}
ii) We have
\begin{align*}
 \F_\Lambda\otimes\{\emptyset,\Omega_W\}&=\sigma(\{A_\Lambda\times A_W, A_\Lambda\in\F_\Lambda, A_W\in\{\emptyset,\Omega_W\}\})\\
 &=\sigma(\{A_\Lambda\times\Omega_W, A_\Lambda\in\F_\Lambda\})\\
 &=\{A_\Lambda\times\Omega_W, A_\Lambda\in\F_\Lambda\},
\end{align*}
which is a sigma-algebra. 
Analogously we show $$\{\emptyset,\Omega_\Lambda\}\otimes\F_W=\{\Omega_\Lambda\times B_W, B_W\in\F_W\}\,.$$ 
Thus using \eqref{Assumption 1}, we get
\begin{align}\label{eq:sigma-algebras}
 \F^\Lambda_T&=\sigma(\{A_\Lambda\times\Omega_W, A_\Lambda\in\F_\Lambda\}\cup\N)\nonumber\\
 \F^W_{R}&=\sigma(\{\Omega_\Lambda\times B_W, B_W\in\F_W\}\cup\N).
\end{align}
The generators of the sigma-algebras in \eqref{eq:sigma-algebras} are $\pi$-systems as they are independent and closed under finite intersection. Let $A\in\{A_\Lambda\times\Omega_W, A_\Lambda\in\F_\Lambda\}\cup\N$ and $B\in\{\Omega_\Lambda\times B_W, B_W\in\F_W\}\cup\N$. If $A\in\N$ (or $B\in\N$) then $\PP(A)=0$ ($\PP(B)=0$) and independence holds. Otherwise, $A=A_\Lambda\times\Omega_W$, for some $A_\Lambda\in\F_\Lambda$ and $B=\Omega_\Lambda\times B_W$, for some $B_W\in\F_W$ and
\begin{align*}
 \PP(A\cap B)&=\PP(A_\Lambda\times B_W)=\int\limits_{A_\Lambda\times B_W}\text{d}\PP_\Lambda\otimes \PP_W=\int\limits_{B_W}\int\limits_{A_\Lambda}\,\text{d}\PP_W\,\text{d}\PP_\Lambda\\
 &=\PP_\Lambda(A_\Lambda)\PP_W(B_W)=\PP(A_\Lambda\times\Omega_W)\PP(\Omega_\Lambda\times B_W)=\PP(A)\PP(B)\,.
\end{align*}
The result now follows from \cite[Lemma 3.6]{kallenberg2006foundations}. 

\par\medskip
{\bf Proof of Proposition \ref{endowment}.} \\
Observe that $\forall t\in[0,T]$, $\Lambda_t$ is $\F^\Lambda_t$-measurable. As $\F^\Lambda_t\subseteq\F^\Lambda_T=(\F_\Lambda\otimes\{\emptyset,\Omega_W\})\vee\N\subseteq\F$\,, then 
 \begin{align*}
  \hat\Lambda_t:=(\Lambda_t,\id):\Omega\rightarrow\IR_+\times\Omega,\quad \omega\mapsto(\Lambda_t(\omega),\omega)
 \end{align*}
 is $\F-\B(\IR_+)\otimes\F$-measurable. Moreover, $W:\IR_+\times\Omega\rightarrow\IR$ is $\B(\IR_+)\otimes\F-\B(\IR)$-measurable. This implies that $M_t=W\circ\hat\Lambda_t$ is $\F-\B(\IR)$-measurable as a composition of measurable functions. Hence $\F^M_t\subseteq\F$ and 
 \begin{align*}
  \F_T=\sigma(\F^M_T\cup\F^\Lambda_T)\subseteq\F
 \end{align*}
 and the statement follows since $\F_t \subset \F_T$\,, for all $t \in [0,T]$\,.

\par\medskip
{\bf Proof of Proposition \ref{M-F-martingale}.}\\
 $(M_t)_{t\in[0,T]}$ is $\FF^M$-adapted and thus $\FF$-adapted. We define $\tilde\Lambda:[0,T]\times\Omega_\Lambda\rightarrow[0,R]$ by $\tilde\Lambda_t(\omega_\Lambda)=\Lambda_t(\omega_\Lambda)$ and $\tilde W:[0,R]\times\Omega_W\rightarrow\IR$ by $\tilde W_t(\omega_W)=W_t(\omega_W)$\,. As $\Lambda$ is bounded, it holds $\sqrt{\tilde\Lambda_T}\in L^1(\Omega_\Lambda,\F_\Lambda,\PP_\Lambda)$. Therefore, using a change of variable formula for Lebesgue-integrals and H\"older's inequality, we get
 \begin{align*}
  \EE[|M_t|]&=\int\limits_{\Omega_\Lambda}\int\limits_{\Omega_W}|\tilde W_{\tilde\Lambda_t(\omega_\Lambda)}(\omega_W)|\PP_W(\domega_W)\PP_{\Lambda_t}(\domega_\Lambda)\\
  &=\int\limits_{\tilde\Lambda_t(\Omega_\Lambda)}\int\limits_{\Omega_W}|\tilde W_{\lambda}(\omega_W)|\PP_W(\domega_W)(\PP_{\Lambda_t}\circ\tilde\Lambda_t^{-1})(\dlambda)\\
  &=\int\limits_{\tilde\Lambda_t(\Omega_\Lambda)}\EE_{\PP_W}[|\tilde W_{\lambda}|](\PP_{\Lambda_t}\circ\tilde\Lambda_t^{-1})(\dlambda)\\
  &\leq\int\limits_{\tilde\Lambda_t(\Omega_\Lambda)}\sqrt{\EE_{\PP_W}[|\tilde W_{\lambda}|^2]}\sqrt{\EE_{\PP_W}[1]}(\PP_{\Lambda_t}\circ\tilde\Lambda_t^{-1})(\dlambda)\\
  &=\int\limits_{\Omega_\Lambda}\sqrt{\tilde\Lambda_t(\omega_\Lambda)}\PP_{\Lambda_t}(\domega_\Lambda)\\
  &\leq \EE[\sqrt{\Lambda_t}]<\infty\,.
 \end{align*}
 Now let $s\leq t$. It follows from Proposition \ref{endowment} in the present paper and \cite[Proposition 6.8]{kallenberg2006foundations} that $\sigma(M_t-M_s,\F^\Lambda_T)\indep\F^M_s$\,. Hence 
 \begin{align}\label{martingale-M-lambda}
  \EE\left[M_t|\F^M_s\vee\F^\Lambda_T\right] = \EE\left[M_s+M_t-M_s|\F^M_s\vee\F^\Lambda_T\right]=M_s+\EE\left[M_t-M_s|\F^\Lambda_T\right]=M_s\,.
 \end{align}
Since $\F_t\subseteq \F^M_t\vee\F^\Lambda_T$\,, then making use of the tower property, of equation \eqref{martingale-M-lambda}, and of the $\FF$-adaptedness of $(M_t)_{t\in[0,T]}$\,, we get  
 \begin{align*}
  \EE[M_t|\F_s]=\EE\left[\EE\left[M_t|\F^M_s\vee\F^\Lambda_T\right]|\F_s\right]=\EE[M_s|\F_s]=M_s
 \end{align*}
 and the result follows.

\medskip
{\bf Proof of Proposition \ref{Hcont}.}\\
{\it Left-continuity.}
 Recall that for any sets $\mathcal{E}_1$, $\mathcal{E}_2$ of subsets it holds: $[\mathcal E_1\subseteq\mathcal E_2\Rightarrow\sigma(\mathcal E_1)\subseteq\sigma(\mathcal E_2)]$. Therefore,
 \begin{align*}
  \sigma\left(\bigcup_{s<t}\F^W_s\right)\subseteq\sigma\left(\bigcup_{s<t}\F^W_s\cup\F^\Lambda_T\right)\text{ and }\F^\Lambda_T\subseteq\sigma\left(\bigcup_{s<t}\F^W_s\cup\F^\Lambda_T\right),
 \end{align*}
 and thus $\sigma\left(\bigcup_{s<t}\F^W_s\right)\cup\F^\Lambda_T\subseteq\sigma\left(\bigcup_{s<t}\F^W_s\cup\F^\Lambda_T\right)$. Now recall that for an arbitrary set $\mathcal{E}$ of subsets and a sigma-algebra $\A$ it holds: $[\mathcal E\subseteq\A\Rightarrow\sigma(\mathcal E)\subseteq\A]$. This, together with the left-continuity of $\FF^W$ implies that
 \begin{align*}
  \Hh_t=\sigma(\F^W_t\cup\F^\Lambda_T)=\sigma\left(\sigma\left(\bigcup_{s<t}\F^W_s\right)\cup\F^\Lambda_T\right)\subseteq\sigma\left(\bigcup_{s<t}\F^W_s\cup\F^\Lambda_T\right)=\bigvee_{s<t}\Hh_s=\Hh_{t-}.
 \end{align*}
 As the inclusion $\Hh_{t-}\subseteq\Hh_t$ is clear, we have that $\HH$ is indeed left-continuous.
 
 {\it Right-continuity.} It follows from \cite[Theorem 1]{wu1982property}. 
 
 {\it Completeness.} It is easy to see that
  
 $$\Hh_t=\F^W_t\vee\F^\Lambda_T=\left(\sigma(W(s), \,s\leq t)\vee\N\right)\vee\F^\Lambda_T=\left(\sigma(W(s), \,s\leq t)\vee\F^\Lambda_T\right)\vee\N\,$$
and the statement is proved.

\par\medskip
{\bf Proof of Proposition \ref{W-BM}.}\\
We know that $W$ \eqref{W-Lambda} is a Brownian motion. In order for $W$ to be an $\HH$-Brownian motion, by \cite[Chapter III, Definition 2.20]{revuz2013continuous} we have to show that for all $s\leq t\leq R$\,, $\sigma(W_t-W_s)\indep\Hh_s$. Let $A\in\Hh_s$, $B\in\B(\IR)$. Then we have
\begin{align*}
 \PP(W_t-W_s\in B,A)&=\EE\left[\EE\left[\ind_{\{W_t-W_s\in B\}}\ind_A|\Hh_s\right]\right]\\
 &=\EE\left[\EE\left[\ind_{\{W_t-W_s\in B\}}|\F^W_s\vee\F^\Lambda_T\right]\ind_A\right].
\end{align*}
Since $\sigma(W_t-W_s,\F^W_s)\subset\F^W_t$ and $\F^W_t\indep\F^\Lambda_T$, for all $t \in [0,T]$\,, then we have
\begin{align*}
 \EE\left[\ind_{\{W_t-W_s\in B\}}|\F^W_s\vee\F^\Lambda_T\right]=\EE\left[\ind_{\{W_t-W_s\in B\}}|\F^W_s\right]=\EE\left[\ind_{\{W_t-W_s\in B\}}\right]=\PP(W_t-W_s\in B)\,.
\end{align*}
Thus
\begin{align*}
 \PP(W_t-W_s\in B,A)=\PP(W_t-W_s\in B)\PP(A)
\end{align*}
and the statement follows.

\end{document}